\documentclass[12pt]{amsart}
\usepackage[colorlinks=true]{hyperref}
\usepackage{fullpage,enumitem,amsmath,amssymb,graphicx, amsthm}
\usepackage{graphicx} % This is a package for including graphics in your solution.
\usepackage{float}
\usepackage{listings}
\usepackage{amsthm}
\usepackage{pdfpages}
\newcommand{\suml}[0]{\sum\limits}
\newcommand{\lr}[1]{\left( #1 \right)}
\newcommand{\bx}[0]{\mathbf{x}}
\newcommand{\bTheta}[0]{\mathbf{\Theta}}

\usepackage{url}

\newtheorem{theorem}{Theorem}[section]
\newtheorem{lemma}[theorem]{Lemma}
\newtheorem{corollary}[theorem]{Corollary}

\theoremstyle{definition}

\theoremstyle{remark}
\newtheorem{remark}[theorem]{Remark}

\numberwithin{equation}{section}

\usepackage{indentfirst}
\usepackage{color}

\setboolean{@twoside}{false} %\includepdf[pages=-, offset=75 -75]{file.pdf}

\usepackage[utf8]{inputenc}
\usepackage{geometry}
\usepackage{algorithm}
\usepackage{algpseudocode}

\begin{document}

\title{A Derandomized Algorithm for RP-ADMM with Symmetric Gauss-Seidel Method}

\author{Jinchao Xu}
\address{Department of Mathematics, Pennsylvania State University, University Park, PA 16802, USA}
%\curraddr{Department of Mathematics and Statistics,
%Case Western Reserve University, Cleveland, Ohio 43403}
\email{xu@math.psu.edu}

\author{Kailai Xu}
\address{Institute for Computational and Mathematical Engineering, Stanford University, CA 94305-4042}
\email{kailaix@stanford.edu}

%    General info
%\subjclass[2000]{Primary 54C40, 14E20; Secondary 46E25, 20C20}

\author{Yinyu Ye}
\address{Department of Management Science and Engineering, Stanford University, CA 94305-4121}
\email{yyye@stanford.edu}

\date{\today}


\keywords{Alternating direction method of multipliers (ADMM), block symmetric Gauss-Seidel}

\begin{abstract}

For multi-block alternating direction method of multipliers(ADMM), where the objective function can be decomposed into multiple block components, we show that with block symmetric Gauss-Seidel iteration, the algorithm will converge quickly. The method will apply a block symmetric Gauss-Seidel iteration in the primal update and a linear correction that can be derived in view of Richard iteration. We also establish the linear convergence rate for linear systems.

\end{abstract}

\maketitle

%\flushleft{
%\textbf{Kailai Xu}\\
%\textbf{\today}
%}
%
%\raggedright\setlength{\parindent}{2em}
%\tableofcontents

\section{Introduction.}

The alternating direction method of multipliers has been very popular recently due to its application in the large-scale problem such as big data problems and machine learning\cite{boyd2011distributed}. Consider the following optimization problem 

\begin{equation}\label{equ:admm}
  \begin{aligned}
  \min &\quad \theta_1(x_1)+\theta_2(x_2)\\
  &\quad A_1x_1+A_2x_2  = b\\
  &\quad x_i\in\chi_i, i=1,2
    \end{aligned}
\end{equation}

Here $\theta_i:\mathbb{R}^n_i\rightarrow \mathbb{R}$ are closed proper convex functions; $\chi_i\subset{\mathbb{R}^n}$ are closed convex sets; $A_i^TA_i$ are nonsingular; the feasible set is nonempty.

We first construct the augmented Lagrangian

\begin{equation}
  L(x_1,x_2,\lambda) = \theta_1(x_1)+\theta_2(x_2) +\lambda^T(A_1x_1+A_2x_2-b) + \frac{\beta}{2} \left\|A_1x_1+A_2x_2-b \right\|^2
\end{equation}

Here $\beta$ is a positive constant. Assume we already have $x^k=(x_1^k,x_2^k),\lambda^k$, now we generate $x^{k+1},\lambda^{k+1}$ as follows

\begin{align*}
 x_1^{k+1} & = \arg\min\limits_{x_1} \left\{\theta_1(x_1) + \frac{\beta}{2}\|A_1x_1+A_2x_2^k  - b\|^2 + (\lambda^k)^T\lr{A_1x_1+A_2x_2^k  - b} \right\} \\
  x_2^{k+1} & = \arg\min\limits_{x_2} \left\{\theta_2(x_2) + \frac{\beta}{2}\|A_1x_1^{k+1}+A_2x_2  - b\|^2 + (\lambda^k)^T\lr{A_1x_1^{k+1}+A_2x_2  - b} \right\}\\
\lambda^{k+1} &= \lambda^k + \beta\left( A_1 x_1^{k+1} + A_2 x_2^{k+2}- b\right)
\end{align*}

It is well known that ADMM for two blocks convergences\cite{he20121}\cite{monteiro2013iteration}\cite{nishihara2015general}. A natural idea is to extend ADMM to more than two blocks, i.e., consider the following optimization problem, 

\begin{equation}\label{equ:admm}
  \begin{aligned}
  \min &\quad \suml_{i=1}^m \theta_i(x_i)\\
  &\quad \suml_{i=1}^m A_ix_i = b\\
  &\quad x_i\in\chi_i, i=1,2,\ldots, m    
  \end{aligned}
\end{equation}

Here $\theta_i:\mathbb{R}^n_i\rightarrow \mathbb{R}$ are closed proper convex functions; $\chi_i\subset{\mathbb{R}^n}$ are closed convex sets; $A_i^TA_i$ are nonsingular; the feasible set is nonempty.

The corresponding augmented Lagrangian function reads

\begin{equation}\label{equ:a1}
  \tilde L^{(1)}(x_1,x_2,\ldots, x_m,\lambda) = \suml_{i=1}^m\theta_i(x_i) + \lambda^T(\suml_{i=1}^m A_ix_i-b) + \frac{\beta}{2}\|\suml_{i=1}^m A_ix_i-b\|^2_2
\end{equation}

Note \eqref{equ:a1} can also be written in form of 

\begin{equation}\label{equ:a2}
  \tilde L^{(2)}(x_1,x_2,\ldots, x_m,\lambda) = \suml_{i=1}^m\theta_i(x_i)  + \frac{\beta}{2}\left\|\suml_{i=1}^m A_ix_i-b +\frac{1}{\beta}\lambda \right\|^2_2 -\frac{1}{2\beta}\|\lambda\|^2_2
\end{equation}

The direct extension of ADMM is shown in Algorithm \ref{algo:direct}.
\begin{algorithm}[H]
\caption{Direct Extension of ADMM}
\flushleft{Assume we already have $x^k,\lambda^k$.}
\label{algo:direct}
\begin{align}
 x_i^{k+1} & = \arg\min\limits_{x_i} \left\{\theta_i(x_i) + \frac{\beta}{2}\|\suml_{j=1}^{i-1}A_j x_j^{k+1} + A_ix_i +\suml_{j=i+1}^{m}A_j x_j^k-b+\frac{1}{\beta}\lambda^k  \|^2  \right\}\quad i=1,2,\ldots,m \label{equ:opt1} \\
 \lambda^{k+1} &= \lambda^k +\beta\left(\suml_{j=1}^m A_j x_j^{k+1} - b\right)\label{equ:opt2}
\end{align}
\end{algorithm}
Unfortunately, such a scheme is not ensured to be convergent\cite{chen2016direct}. For example, consider applying ADMM with three blocks to the following problem,

\begin{equation}\label{equ:counter}
  \begin{aligned}
      \min &\quad 0\\
      \mbox{s.t.} &\quad Ax = \mathbf{0}
  \end{aligned}
\end{equation}
where

\begin{equation}
  A =(A_1,A_2,A_3)= \begin{pmatrix}
      1 & 1& 1\\
      1 & 1 & 2\\
      1 & 2 & 2
  \end{pmatrix}
\end{equation}

Then Algorithm \ref{algo:direct} gives a linear system where the spectral radius of the iteration matrix can be shown to be greater than $1$.

As a remedy for the divergence of multi-block ADMM, \cite{sun2015expected} proposes a randomized algorithm(RP-ADMM) and proved its convergence for linear objective function by an expectation argument. The algorithm for computing $x^{k+1},\mu^{k+1}$ from $x^{k},\mu^{k}$ reads as Algorithm \ref{algo:rp-admm}.

\begin{algorithm}
\caption{RP-ADMM}
\label{algo:rp-admm}
\begin{enumerate}
    \item Pick a permutation $\sigma$ of $\{1,2,\ldots, m\}$ uniformly at random.
    \item For $i=1,2,\ldots,m$, compute $x^{k+1}_{\sigma(i)}$ by
    
    \begin{equation}
  x_{\sigma(i)}^{k+1} = \arg\min_{x_{\sigma(i)\in \chi_{\sigma(i)}}}\tilde L^{(2)}(x_{\sigma(1)}^{k+1} , \ldots, x_{\sigma(i-1)}^{k+1}, x_{\sigma(i)},x_{\sigma(i+1)}^{k+1},\ldots, x_{\sigma(m)}^{k+1};\lambda^k)
\end{equation}

\item 
\begin{equation}
  \lambda^{k+1} = \lambda^k + \beta\left(\suml_{i=1}^m A_j x_j^{k+1} - b\right)
\end{equation}

\end{enumerate}
\end{algorithm}

There are other efforts to improve the convergence of multi-block ADMM. For example, in \cite{he2012alternating}, the author suggests a correction after $x^{k+1},\mu^{k+1}$ obtained from Algorithm \ref{algo:direct}. To make comparison between different algorithms, we adapt the algorithm ADM-G in \cite{he2012alternating} to the form in Algorithm \ref{algo:admmg}.

\begin{algorithm}[H]
\caption{The ADMM with Gaussian back substitution(G-ADMM)}
\label{algo:admmg}
Let $\alpha\in(0,1),\beta>0$.
\begin{enumerate}
    \item Prediction step.
    \begin{align}
 \tilde x_i^{k} & = \arg\min\limits_{x_i} \left\{\theta_i(x_i) + \frac{\beta}{2}\|\suml_{j=1}^{i-1}A_j \tilde x_j^{k} + A_ix_i +\suml_{j=i+1}^{m}A_j x_j^k-b+\frac{1}{\beta}\lambda^k  \|^2  \right\}\\ 
 & \qquad\qquad\qquad\qquad\qquad\qquad\qquad\qquad\qquad\qquad i=1,2,\ldots,m \nonumber 
 \end{align}
\item Correction step.

Assume $A^TA=D+L+L^T$, where $L$ is lower triangular matrix and $D$ is a diagonal matrix.
\begin{equation}\label{equ:correction}
  \left\{ 
  \begin{aligned}
  x^{k+1} &= x^k + \alpha (D+L)^{-1}D(\tilde x^k-x^k)\\
  \lambda^{k+1}     &= \lambda^k +\alpha \beta\left(\suml_{i=1}^m A_j \tilde x_j^{k} - b\right)
  \end{aligned}
 \right.
\end{equation}

\end{enumerate}
\end{algorithm}

In this paper, we propose an approach based on the insights into the two different algorithms described above(RP-ADMM and G-ADMM). We attribute the convergence of RP-ADMM to the symmetrization of the optimization step. Indeed, when we randomly permute the order of optimization variables, it is actually symmetrizing the update procedure for $x^k$. In Algorithm \ref{algo:direct}, we may have totally different convergence behavior if we exchange the order of variables $x_1,x_2,\ldots,x_m$, for example, we may get a convergence scheme if we do update $x_1\rightarrow x_2\rightarrow \ldots\rightarrow x_m$, and a divergence scheme if we do update $x_m\rightarrow x_{m-1}\rightarrow \ldots\rightarrow x_1$. This is not desirable and Algorithm \ref{algo:rp-admm} overcomes this by doing a random `shuffling'. Based on this observation, we propose a Symmetric Gauss-Seidel ADMM(S-ADMM) scheme for the multi-block problem and prove its convergence in the worst case, i.e., the objective function is $0$(not strongly convex).

Meanwhile, in G-ADMM, the authors used a correction step after each loop. In our algorithm, we will use Richard iteration\cite{jinchaoxulecture} to `correct' the dual variables in the Schur decomposition, which turns out to be the gradient descent for dual variables in the augmented Lagrangian method. We remark that although viewed in iteration methods, Richard iteration is not the best method to solve $Ax=b$, but still, it outperforms Algorithm \ref{algo:admmg} in many numerical examples we did.

%Meanwhile, in G-ADMM, we think that the $x$ update in \eqref{equ:correction} is actually forcing $x^{k+1}$ not to diverge too much from the constraint $Ax=b$. However, this is not the best method viewed in this fashion. Instead, we would do the Richard iteration\cite{jinchaoxulecture} in this step. We remark that although viewed in iteration methods, Richard iteration is not the best method to solve $Ax=b$, but still, it outperforms Algorithm \ref{algo:admmg} in many numerical examples we did.

The rest of the paper is organized as follows. In Section \ref{sect:v-cycle}, we describe S-ADMM algorithm and prove its convergence in the worst case. In Section \ref{sect:numerical}, we apply the algorithm to solve some concrete examples and compare it with the existing methods numerically. 

\section{S-ADMM.}\label{sect:v-cycle}

We propose the following algorithm to solve the problem.

\begin{algorithm}
\caption{S-ADMM}    
\label{algo:v-cycle}

Let $\beta>0,\omega\in(0,2\beta)$. Assume we already have $x^k$.

\begin{enumerate}
    \item Forward optimization.
    
    \begin{align}
 \tilde x_i^{k} & = \arg\min\limits_{x_i} \left\{\theta_i(x_i) + \frac{\beta}{2}\|\suml_{j=1}^{i-1}A_j \tilde x_j^{k} + A_ix_i +\suml_{j=i+1}^{m}A_j x_j^k-b+\frac{1}{\beta}\lambda^k  \|^2  \right\}\\ 
 & \qquad\qquad\qquad\qquad\qquad\qquad\qquad\qquad\qquad\qquad i=1,2,\ldots,m \nonumber
\end{align}
\item Backward optimization.
\begin{align}
  x_i^{k+1} & = \arg\min\limits_{x_i} \left\{\theta_i(x_i) + \frac{\beta}{2}\|\suml_{j=1}^{i-1}A_j \tilde x_j^{k} + A_ix_i +\suml_{j=i+1}^{m}A_j  x_j^{k+1}-b+\frac{1}{\beta}\lambda^k  \|^2  \right\}\\ 
 & \qquad\qquad\qquad\qquad\qquad\qquad\qquad\qquad\qquad\qquad i=m-1,m-2,\ldots,1 \nonumber 
\end{align}
\item Dual update.

\begin{equation}
  \lambda^{k+1} = \lambda^k +\omega\left(\suml_{j=1}^m A_j x_j^{k+1} - b\right)
\end{equation}

\end{enumerate}
\end{algorithm}

Consider the following optimization problem,

\begin{align}
    \min\limits_{x_i,i=1,2,\ldots,m} &\quad \suml_{i=1}^m \frac{1}{2}\theta_i x_i^2\\
    s.t. &\quad \suml_{i=1}^m A_ix_i = b
\end{align}

Here $\theta_i\geq 0$. The augmented Lagrangian is 

\begin{equation}
  L(x_1,x_2,\ldots,x_m;\theta_1,\theta_2,\ldots, \theta_m) = \frac{1}{2}\suml_{i=1}^m \theta_i x_i^2-\mu^T\left(\suml_{i=1}^m A_ix_i-b\right) + \frac{\beta}{2}\left\|\suml_{i=1}^m A_ix_i-b\right\|^2
\end{equation}

The optimization problem is equivalent to solve

\begin{equation}
  \frac{\partial L}{\partial x_i} = \frac{\partial L}{\partial \theta_i} = 0, i=1,2,\ldots,m
\end{equation}

i.e.

\begin{equation}\label{equ:more}
  \begin{pmatrix}
\theta_1+\beta A_1^TA_1 & \beta A_1^TA_2 & \ldots & \beta A_1^TA_m  & -A_1^T \\ 
\beta A_2^TA_1  & \theta_2+\beta A_2^TA_2  & \ldots &\beta A_2^TA_m  &-A_2^T \\ 
\vdots &\vdots  &\vdots  & \vdots &\vdots \\ 
\beta A_m^TA_1  &\beta A_m^TA_2  & \ldots & \theta_m+\beta A_m^TA_m  & -A_m^T \\ 
 -A_1&-A_2  &\ldots  &-A_m  & 0
\end{pmatrix}\begin{pmatrix}
x_1\\ 
x_2\\ 
\vdots\\ 
x_M\\ 
\mu
\end{pmatrix}=\begin{pmatrix}
\beta A_1^Tb\\ 
\beta A_2^Tb\\ 
\vdots\\ 
\beta A_m^Tb\\ 
-b
\end{pmatrix}
\end{equation}

Let 

\begin{equation}
  \theta = \begin{pmatrix}
\theta_1 &  &  & \\ 
 & \theta_2 &  & \\ 
 &  &\ddots  & \\ 
 &  &  & \theta_m
\end{pmatrix}
\end{equation}

and $A=(A_1 A_2 \ldots A_m)$, $G = \theta + \beta A^TA$, then (\ref{equ:more}) is equivalent to 

\begin{equation}\label{equ:mat}
  \begin{pmatrix}
G & -A^T \\ 
-A  & 0 
\end{pmatrix}\begin{pmatrix}
x\\ 
\mu
\end{pmatrix}=
\begin{pmatrix}
\beta A^Tb\\ 
-b\end{pmatrix}
\end{equation}

Now we assume $G$ is invertible. Note this is true if we assume $\theta_i>0,\forall i$, or $A^TA$ is nonsingular. Or most generally, let $A^TA = U^T\Lambda U$, where $U$ is orthogonal matrix, we have $G = U^T(\theta + \beta \Lambda)U$. We assume $\theta_i>0$ or $\Lambda_{ii}>0$ for all $i=1,2,\ldots,m$.

If we do Schur decomposition, we obtain

\begin{equation}\label{equ:schur}
  \begin{pmatrix}
G & -A^T \\ 
0  & AG^{-1}A^T 
\end{pmatrix}\begin{pmatrix}
x\\ 
\mu
\end{pmatrix}=
\begin{pmatrix}
\beta A^Tb\\ 
b-\beta AG^{-1}A^Tb
\end{pmatrix}
\end{equation}

The well-known augmented Lagrangian method solves (\ref{equ:schur}) by doing Gauss-Seidel iteration\cite{jinchaoxulecture}\cite{jinchaoxulecture2} , i.e.

\begin{align}
    Gx^{n+1} - A^T \mu^n & = \beta A^Tb\label{equ:1} \\
    \mu^{n+1} &= \mu^n + \omega (b-\beta AG^{-1}A^T b - AG^{-1}A^T \mu^n)\label{equ:2}
\end{align}

Note (\ref{equ:2}) is exactly Richardson iteration for $AG^{-1}A^T\mu = b-\beta AG^{-1}A^T b$. Due to (\ref{equ:1}), we have

\begin{equation}
  x^{n+1} = G^{-1}A^T \mu^n + \beta G^{-1}A^Tb
\end{equation}

Then

\begin{equation}
  Ax^{n+1} = AG^{-1}A^T \mu^n + \beta AG^{-1}A^Tb
\end{equation}

(\ref{equ:2}) is exactly

\begin{equation}
  \mu^{n+1} = \mu^n + \omega (b- Ax^{n+1})
\end{equation}
which is consistent with augmented Lagrangian method.

We think the key here is to symmetrize the iteration process for (\ref{equ:1}). Therefore, we propose the symmetric Gauss-Seidel method for (\ref{equ:1}).

Let $G=L+L^T+D$, where $L$ is a lower triangular matrix and $D$ is a diagonal matrix.

\begin{align*}
    (L+D)x^{n+\frac{1}{2}} & = -L^Tx^n + A^T \mu^n + \beta A^T b\\
    (L^T+D)x^{n+1} & = -Lx^{n+\frac{1}{2}} + A^T\mu^n + \beta A^T b
\end{align*}

We have
\begin{equation}
 x^{n+1} = x^n + \tilde G^{-1}(A^T\mu^n + \beta A^T b-Gx^n)
\end{equation}
where $\tilde G =  (L+D)D^{-1}(L^T+D)$.

Instead of solving (\ref{equ:2}) directly, we substitute $G$ by $\tilde G$

\begin{align}
    \mu^{n+1} &= \mu^n + \omega (b-Ax^{n+1})
\end{align}

We will later prove that the scheme converges.

In sum, the scheme is 

\begin{align}
    (L+D)x^{n+\frac{1}{2}} & = -L^Tx^n + A^T \mu^n + \beta A^T b\label{equ:c1} \\
    (L^T+D)x^{n+1} & = -Lx^{n+\frac{1}{2}} + A^T\mu^n + \beta A^T b\label{equ:c2}\\
    \mu^{n+1} &=\mu^n + \omega (b-Ax^{n+1})\label{equ:4}
\end{align}

or in the compact form

\begin{equation}\label{equ:58}
  \begin{pmatrix}
x^{n+1}\\ 
\mu^{n+1}
\end{pmatrix}
=
\begin{pmatrix}
x^{n}\\ 
\mu^{n}
\end{pmatrix} + \begin{pmatrix}
- \tilde G & 0 \\ 
\omega A  & I 
\end{pmatrix}^{-1}\left(
\begin{pmatrix}
- \beta A^T b\\ 
\omega b
\end{pmatrix}-
\begin{pmatrix}
- G &  A^T\\ 
\omega A & 0 
\end{pmatrix}
\begin{pmatrix}
x^n\\ 
\mu^n
\end{pmatrix}
 \right)
\end{equation}

We now select appropriate $\omega$ such that 

\begin{equation}\label{equ:rho}
  \rho\left(I-\begin{pmatrix}
- \tilde G & 0 \\ 
\omega A  & I 
\end{pmatrix}^{-1}\begin{pmatrix}
- G &  A^T\\ 
\omega A & 0
\end{pmatrix} \right)<1
\end{equation}

%or equivalently
%
%\begin{equation}\label{equ:rho}
%  \lambda\left(\begin{pmatrix}
%- \tilde G & O \\ 
%\omega A  & I 
%\end{pmatrix}^{-1}\begin{pmatrix}
%- G &  A^T\\ 
%\omega A & O 
%\end{pmatrix} \right)\in(0,2)
%\end{equation}
%

\begin{theorem}\label{equ:powerful}
    Assume $H=\tilde G^{-1}G$ satisfies $\lambda(H)\in(0,1)$ and $G = \beta A^TA$ is invertible. If $0<\omega<2\beta$, then (\ref{equ:rho}) holds.
\end{theorem}

\begin{proof}
    \begin{equation}
  \lambda\left(\begin{pmatrix}
- \tilde G & 0 \\ 
\omega A  & I 
\end{pmatrix}^{-1}\begin{pmatrix}
- G &  A^T\\ 
\omega A & 0 
\end{pmatrix} \right) = \begin{pmatrix}
    \tilde G^{-1}G & -\tilde G^{-1}A^T \\
    -\omega A\tilde G^{-1}G + \omega A & \omega A\tilde G^{-1}A^T
\end{pmatrix}
\end{equation}

Assume $\lambda$ is an eigenvalue of the above matrix and $\begin{pmatrix}
    x\\
    y
\end{pmatrix}$ is the corresponding eigenvector, we then have

\begin{equation}
  \begin{pmatrix}
    \tilde G^{-1}G & -\tilde G^{-1}A^T \\
    -\omega A\tilde G^{-1}G + \omega A & \omega A\tilde G^{-1}A^T
\end{pmatrix}\begin{pmatrix}
    x\\
    y
\end{pmatrix} = \lambda \begin{pmatrix}
    x\\
    y
\end{pmatrix}
\end{equation}

which is equivalent to

\begin{align}
    Hx - Hz &= \lambda x\label{equ:x1} \\
    -\omega KHx + \omega Kx + \omega KHz &= \lambda z\label{equ:x2}
\end{align}

where $z=G^{-1}A^Ty,K = G^{-1}A^TA=\frac{1}{\beta}I, H = \tilde G^{-1}G$. Let $c=\frac{\omega}{\beta}$, it is easy to obtain

\begin{equation}
  c(1-\lambda)x = \lambda z
\end{equation}

From this we can see if $\lambda=0$, then $x=0$, and from (\ref{equ:x1}) we have $Hz=0$, as $H$ is nonsingular, we have $z=0$, a contradiction. Thus $\lambda\neq 0$ and

\begin{equation}
  z = \frac{c(1-\lambda)}{\lambda}x
\end{equation}

Plug it into (\ref{equ:x1}), we have

\begin{equation}
  Hx = \frac{\lambda}{1-\frac{c(1-\lambda)}{\lambda}}x
\end{equation}

This indicates 

\begin{equation}\label{equ:ne}
  0<\frac{\lambda}{1-\frac{c(1-\lambda)}{\lambda}}<1
\end{equation}

Now we estimate $\lambda$, let

\begin{equation}
  \frac{\lambda}{1-\frac{c(1-\lambda)}{\lambda}} = \xi\in(0,1)
\end{equation}
We have

\begin{align}
\lambda^2 -(\xi+c\xi)\lambda + c\xi &=0\\
  \lambda &= \frac{\xi(1+c)\pm \sqrt{\xi^2(1+c)^2-4\xi c}}{2}\label{equ:rootsol}
\end{align}

\begin{enumerate}
    \item If $\xi^2(1+c)^2-4\xi c\geq 0$, i.e. $\xi\geq \frac{4c}{(1+c)^2}$, $\lambda$ will be real. Note in this case $c\neq 1$. We have from the first inequality in (\ref{equ:ne})
    
    \begin{equation}
  \lambda{ \geq \frac{c}{c+1}}
\end{equation}

and from the second     
    \begin{equation}
  (\lambda - c)(\lambda -1)<0
\end{equation}

On condition that $c<2,i.e. \omega<2\beta$, we have $0<\lambda<2$, and thus the statement holds.

\item If $\xi^2(1+c)^2-4\xi c< 0$,i.e., $\xi< \frac{4c}{(1+c)^2} $, $\lambda$ will be a complex number. And we have
\begin{equation}
 |\lambda - 1| = \frac{1}{2}\sqrt{(\xi(1+c)-2)^2 + 4\xi c - \xi^2(1+c)^2}=\sqrt{1-\xi}<1 
\end{equation}

Thus, (\ref{equ:rho}) holds.

\end{enumerate}

\end{proof}

\begin{remark}
    We would like to point out the relationship between the iteration matrix
    
    \begin{equation}
  \begin{pmatrix}
    I-\tilde G^{-1}A^TA & \tilde G^{-1}A^T \\
     A\tilde G^{-1}A^TA -  A &  A\tilde G^{-1}A^T
\end{pmatrix}
\end{equation}
when $\omega=1$. In \cite{sun2015expected}, they discuss the eigenvalues of 

\begin{equation}
  M = \begin{pmatrix}
    I-QA^TA & QA^T \\
     AQA^TA -  A &  AQA^T
\end{pmatrix}
\end{equation}
and proved that 

\begin{equation}
  \frac{(1-\lambda)^2}{1-2\lambda}\in  \mathrm{eig}(QA^TA) \Leftrightarrow \lambda \in \mathrm{eig}(M)
\end{equation}

Based on this observation, it is proved that if $\mathrm{eig}(QA^TA)
\in(0,\frac{4}{3})$, then $\mathrm{eig}(M)\in (-1,1) $. In our example, $Q=\tilde G^{-1}$ when $\omega=1$. Note

\begin{equation}\label{equ:eig}
  \mathrm{eig}(QA^TA) = \mathrm{eig}(\tilde G^{-1}G)\in(0,1)
\end{equation}

However, this does not mean the convergence behavior is better for S-ADMM, as we can only obtain $\mathrm{eig}(M)\in (-1,1) $ from \eqref{equ:eig}.

\end{remark}

We can now actually prove that the scheme (\ref{equ:c1})(\ref{equ:c2})(\ref{equ:4}) converges.

\begin{lemma}
    The eigenvalues of $\tilde G^{-1}G$ are distributed in $(0,1)$.
\end{lemma}

\begin{proof}

Note 

\begin{equation}
  \mathrm{eig}(G^{-1}LD^{-1}L^T) = \mathrm{eig}(\tilde G^{-\frac{1}{2}}LD^{-1}L^T\tilde G^{-\frac{1}{2}})
\end{equation}

and therefore the eigenvalues of $G^{-1}\tilde G$ are all greater than $1$. Thus the eigenvalues of $\tilde G^{-1}G$ are distributed in $(0,1)$

\end{proof}

\begin{corollary}
    If $0<\omega<2\beta$, then the scheme (\ref{equ:c1})(\ref{equ:c2})(\ref{equ:4}) converges for $\theta = \mathbf{0}$ and full column rank $A$.
\end{corollary}

\newpage

\section{Numerical Examples}\label{sect:numerical}

In this section, we apply the S-ADMM to solve the problem of counterexample proposed in \cite{chen2016direct} and also a quadratic objective function with 1-norm penalty and linear constraint. We compare this method with Algorithm \ref{algo:rp-admm} and Algorithm \ref{algo:admmg}.

The code is written in Python and run on an x86\_64 Linux machine.

\subsection{Counter-example in \cite{chen2016direct}}

The problem is presented in \eqref{equ:counter}. It is analyzed in \cite{sun2015expected} that a cyclic ADMMM is a divergent scheme, and in the same paper the author proved that Algorithm \ref{algo:rp-admm} convergences with high probability. 

To make the result comparable to each other, we pick the same initial points $b=(0,0,0)$ and $x_0=(1,1,1)$, with $\beta=4.0$ and $\alpha=0.2$.

The result is presented in Figure \ref{fig:1}. We see that S-ADMM performs as well as PR-ADMM algorithm, but it is more oscillating compared to G-ADMM.

\begin{figure}[H] % \usepackage{float}
\centering

\includegraphics[width=0.8\textwidth,keepaspectratio]{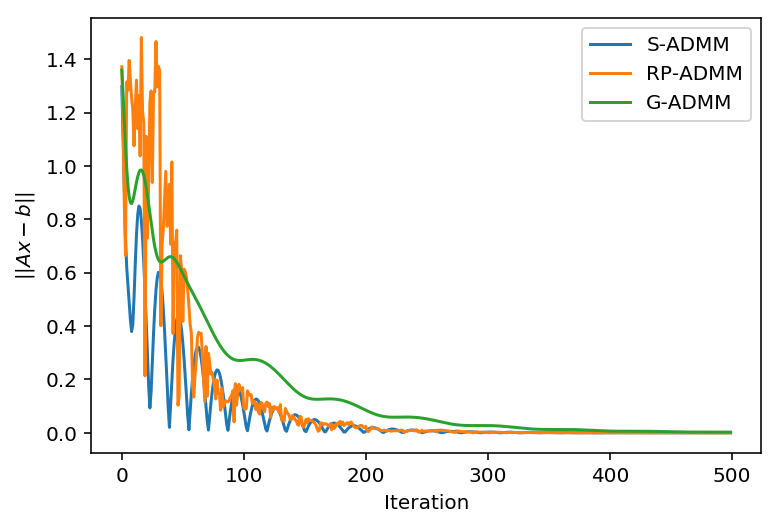}
\caption{Counter example, the curves describes $\|Ax^k-b\|$ in the iterations.}
\label{fig:1}
\end{figure}

\subsection{Quadratic Objective Function with 1-norm Penalty.} In this example, we solve the following manufactured optimization problem.

\begin{equation}
  \begin{aligned}
  \min &\quad \suml_{i=1}^{10} \bx^T\Theta_i\bx + |\bx|_1\\
  \mbox{s.t.} &\quad \suml_{i=1}^{10} A_i\bx_i = \mathbf{b}
  \end{aligned}
\end{equation}

Here

\begin{equation}
  \bTheta_i = \begin{pmatrix}
      5+i& 1\\
      1& 5+i
  \end{pmatrix}, \mathbf{b} = \begin{pmatrix}
      1\\
      11\\
      1
  \end{pmatrix}, A_i = \begin{pmatrix}
      1+i & 2+i\\
      3+i & 4+i\\
      \ldots &\ldots \\
      19+i & 20+i
  \end{pmatrix}
\end{equation}

We run the algorithms for 500 iterations and obtain the result shown in Figure \ref{fig:2} and Figure \ref{fig:3}.

\begin{figure}[H] % \usepackage{float}
\centering
\includegraphics[width=0.8\textwidth,keepaspectratio]{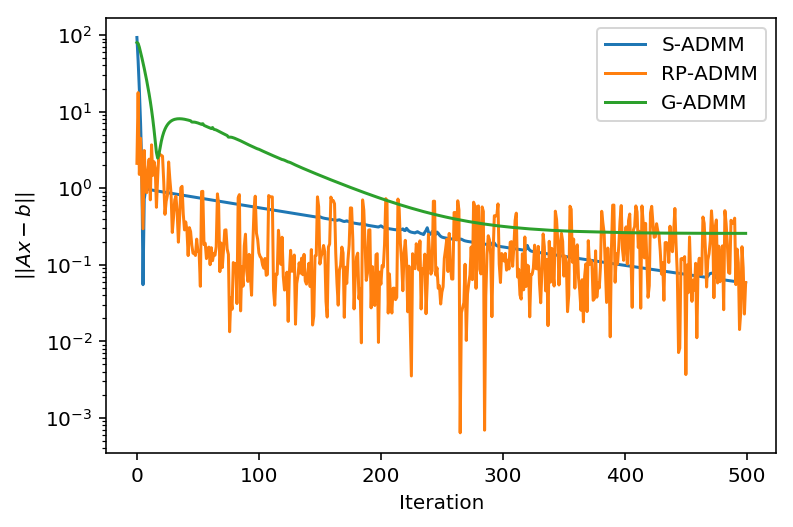}
\caption{$\|Ax^k-b\|$ in each iteration}
\label{fig:2}
\end{figure}

\begin{figure}[H] % \usepackage{float}
\centering
\includegraphics[width=0.8\textwidth,keepaspectratio]{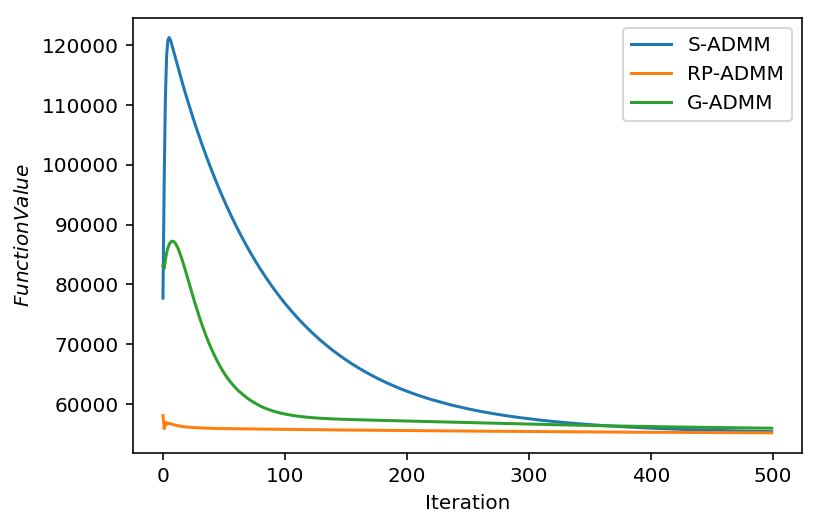}
\caption{Function value in each iteration}
\label{fig:3}
\end{figure}

We see in Figure \ref{fig:2} that the primal feasibility $\|Ax^k-b\|$ of S-ADMM decreases continuous and linearly, while for G-ADMM the residual tends to stop decreasing after 300 iterations. For PR-ADMM, the residual oscillates and decreases slowly. We have to remark that during the numerical experiments the author observes not every time PR-ADMM converges. The program may halt due to float overflow. This is easy to understand because in \cite{sun2015expected} the authors proves that the algorithm converges using the expectation argument. This argument only guarantees that PR-ADMM converges with high probability.

However, if viewed in terms of function value decreasing, S-ADMM may not seem very inviting. It is the slowest among all.

\section{Conclusions}

We propose a new algorithm based on block symmetric Gauss-Seidel algorithm to do the primal update in ADMM. The algorithm will converge if the objective is strongly convex and when the objective function is $0$, the rate can actually be found with linear algebra. This algorithm can be viewed as a de-randomized version of PR-ADMM and the dual update can be viewed as a Richard iteration correction in the Schur decomposition. Moreover, we interpret the algorithm as adding a regularization or proximal term to the original problem and then solve the problem analytically. This facilitates us to see how the algorithm helps accelerate the convergence of ADMM. We believe that a better correction step could be found to accelerate the algorithm, by doing a more sophisticated correction step.

\newpage

\bibliographystyle{plain}
\bibliography{ref}
\end{document}